%% file: GaussRadauRevised.tex
\RequirePackage{fix-cm}
\documentclass[smallextended]{svjour3} 
\usepackage{graphicx}

\usepackage{amsmath,amssymb,amsfonts,fancybox,shadow,color}
\usepackage{hyperref,todonotes}
\usepackage{booktabs}
\usepackage{xcolor,bm}
\usepackage{enumitem} 
\usepackage{subcaption}
\usepackage{verbatim}
\usepackage{algorithm}
\usepackage{algpseudocode}
\usepackage[normalem]{ulem}

\include{mathmacros}
\smartqed 
\renewenvironment{proof}{{\it Proof.}}{\hfill\qedsymbol}
\newcommand{\qedsymbol}{$\square$}

\newcommand{\CC}{{\mathbb{C}}}

\def\be#1\ee{\begin{equation}#1\end{equation}}

\newcommand\hgam{\boldsymbol{\hat\gamma}}
\renewcommand\gam{\boldsymbol{\gamma}}

\newcommand{\Alpha}{\boldsymbol \alpha}
\renewcommand{\Beta}{\boldsymbol \beta}

\newcommand{\hKappa}{\hat{\boldsymbol \kappa}}

\renewcommand{\Xi}{\boldsymbol \xi}
\newcommand\blF{ {\cal F}}
\newcommand\bbF{ {\mathbf{F}}}

\newcommand\blTF{ {\tilde {\cal F}}}
\newcommand\blC{{\cal C}}

\newcommand\blU{{\cal U}}

\newcommand{\blUi}[2]{\blU_{#1}^{#2}}

\def\bm{\mathbf{m}}
\def\bT{\mathbf{T}}

\DeclareMathOperator*{\argmin}{arg\,min}

\DeclareFontFamily{U}{mathx}{\hyphenchar\font45}
\DeclareFontShape{U}{mathx}{m}{n}{
 <5> <6> <7> <8> <9> <10>
 <10.95> <12> <14.4> <17.28> <20.74> <24.88>
 mathx10
 }{}
\DeclareSymbolFont{mathx}{U}{mathx}{m}{n}
\DeclareFontSubstitution{U}{mathx}{m}{n}
\DeclareMathAccent{\widecheck}{0}{mathx}{"71}
\DeclareMathAccent{\wideparen}{0}{mathx}{"75}

\definecolor{red}{rgb}{0,0,0}
\definecolor{purple}{rgb}{0,0,0}

\begin{document}

\title{A Krylov projection algorithm for large symmetric matrices with dense spectra}

\titlerunning{Krylov projection for large matrices} 

\author{
 Vladimir Druskin \and J\"orn Zimmerling
}


\institute{J\"orn Zimmerling, \at
 Institutionen f\"or informationsteknologi, Uppsala Universitet, L\"agerhyddsv\"agen 2, 75237 Uppsala, Sweden. \\
 \email{jorn.zimmerling@it.uu.se} 
 \and
 Vladimir Druskin, \at
 Worcester Polytechnic Institute, USA.
		\email{vdruskin@wpi.edu} 
}

\date{\today}

\maketitle

\begin{abstract}
We consider the approximation of $B^T (A+sI)^{-1} B$ for large s.p.d. $A\in\RR^{n\times n}$ with dense spectrum and $B\in\RR^{n\times p}$, $p\ll n$. We target the computations of Multiple-Input Multiple-Output (MIMO) transfer functions for large-scale discretizations of problems with continuous spectral measures, such as linear time-invariant (LTI) PDEs on unbounded domains. Traditional Krylov methods such as the Lanczos or CG algorithm are known to be optimal for the computation of $(A+sI)^{-1}B$ with real positive $s$, resulting in an adaptation to the distinctively discrete and nonuniform spectra. However, the adaptation is damped for matrices with dense spectra. It was demonstrated in \cite{zimmerling2025monotonicity} that averaging Gau{\ss} and Gau\ss -Radau quadratures computed using the block-Lanczos method significantly reduces approximation errors for such problems. Here we introduce an adaptive Kreĭn-Nudelman extension to the (block) Lanczos recursions, allowing further acceleration at negligible $o(n)$ cost. Similar to the Gau\ss -Radau quadrature, a low-rank modification is applied to the (block) Lanczos matrix. However, unlike the Gau\ss -Radau quadrature, this modification depends on $\sqrt{s}$ and can be considered in the framework of the Hermite-Pad\'e approximants, which are known to be efficient for problems with branch-cuts, that can be good approximations to dense spectral intervals. 
Numerical results for large-scale discretizations of heat-diffusion and quasi-magnetostatic Maxwell’s operators in unbounded domains confirm the efficiency of the proposed approach. 

\end{abstract}

\section{Introduction}

\subsection{Problem statement}

Let $A=A^T\in\RR^{n\times n}$ be a symmetric positive definite (s.p.d.) matrix and let $B$ be a tall matrix with $B\in\RR^{n\times p}$, with $p\ll n$. We refer to $p$ as the block size. We want to approximate the square multi-input multi-output (MIMO) transfer function
\be \label{eq:prob1}
\blF(s) =B^T(A+sI)^{-1}B,
\ee
for $s\in\CC$ outside the spectrum $A$. Generally, $\blF(s)\in\CC^{p\times p}$ is a complex symmetric matrix-valued function and it becomes s.p.d. for real positive $s$.
We compute \eqref{eq:prob1}
via approximations of $(A+sI)^{-1}B$ on a block Krylov subspace 
\[
\mathcal{K}_m(A,B)={\rm blkspan} \{B, AB, A^2 B , \dots A^{m-1} B \}
\]
where $ \rm blkspan$ means that the whole space range $([B, AB, A^2 B , \dots A^{m-1} B])$ is generated.
To simplify notation, without loss of generality, we assume that the matrix $B$ has orthonormal columns.
Conventionally, the (block) Krylov subspace approximations of \eqref{eq:prob1} are computed via (block)-Gau\ss ian quadratures, e.g., \cite{GM10,RRT16,lot2013}. The efficiency of such an approach follows from the optimality of the Lanczos approximations for $s\in\RR_+$ yielding adaptation to coarse nonuniform spectral distributions, however, this property weakens on intervals of dense uniform spectra. An error improvement of about one order of magnitude for a given $m$ was obtained in \cite{zimmerling2025monotonicity} by averaging Gau{\ss} and Gau\ss -Radau quadratures. Combinations of such quadratures were used for tight two-sided bounds \cite{lot2008,Meurant2023}, such that averaging naturally leads to a more accurate estimate. In \cite{zimmerling2025monotonicity} this approach was rigorously justified for the SISO case for problems with a continuous spectrum. 

The derivation of block-Gau\ss -Radau formulas in \cite{zimmerling2025monotonicity} was based on the connection of the block-Lanczos algorithm to the block extension of discrete Stieltjes strings and block-Stieltjes continued fraction. The scalar variant of this connection was introduced in the seminal work by Mark Kreĭn in the 1950s,.e.g., see \cite{Krein1967,Krein1952,Krein1947}. An extension of this approach for problems with a continuous spectrum was introduced by Kreĭn and Nudelman \cite{KreinNudelman1989} by adding a damper at the end of the Stieltjes string. Here we use this approach by adaptively modifying the block-Lanczos recursion to move poorly convergent Ritz value to the other Riemann sheet, while still matching the same number of spectral moments as in the conventional block-Lanczos method. The intuition behind such a modification is the description of continuous spectra via the scattering poles or resonances in unbounded domains \cite{Zworski}. Such poles give a compact representation of the dense interval of the spectral measure allowing spectral adaptation where classical Krylov methods fail. A mathematical foundation explanation behind this phenomenon is that the conventional quadrature algorithms are the Pad\'e approximations while the Kreĭn-Nudelman extension is closely related to the Hermite-Pad\'e approximation, and the latter is known to be superior to the former for problems with branch cuts.  We focus on applications stemming from diffusive electromagnetic problems in unbounded domains.

The remainder of the article is organized as follows: we review the basic properties of the block-Gau{\ss} quadrature computed via the block-Lanczos Algorithm, its connection to Stieltjes strings and representation via truncated Stieltjes-matrix continued fraction in section~\ref{sec:BlLanc}; In section~\ref{sec:Krein-Nudelman1} we extend block Stieltjes string and continued fraction to, respectively, Kreĭn-Nudelman strings and continued fractions, and then revert back the Kreĭn-Nudelman extension in terms of the Lanczos block-tridiagonal matrix. Then we examine the adaptive choice of the Kreĭn-Nudelman parameters, optimizing the regularity of the spectral density, show that Gau{\ss} and Gau\ss -Radau quadratures are two limiting cases of the Kreĭn-Nudelman extension, and derive a two-sided error bound for the latter. In section~\ref{sec:Hermite-Pade} we develop an adaptive spectral optimization algorithm for determining Kreĭn-Nudelman parameters. Then we bound the error of our approximation via the product of the Chebyshev error and the error of measure approximation on the spectral interval, which shows that using a continuous spectrum in our approximant accelerates convergence if the underlying problem has a continuous spectrum. The adaptive spectral weight used in optimization functional allows us to extend the reasoning (to some degree) to non-uniformly distributed discrete spectral measures with clustered intervals. The derivations of this section are done in a general matrix function framework, thus showing that both our algorithm and analysis can be extended from \eqref{eq:prob1} to the regular on $A$' s spectrum arbitrary matrix functions. Last, numerical examples are presented in section~\ref{sec:NumEx}.

\subsection{Notation}\label{sec:note}
Given two square symmetric matrices $G_1, G_2$ we use the notation
$G_1<G_2$ to mean that
the matrix $G_2-G_1$ is positive definite. 
A sequence of matrices $\{G_m\}_{m\ge 0}$ is said to be monotonically
increasing (resp. decreasing) if $G_m < G_{m+1}$ (resp. $G_{m+1}<G_m$)
for all $m$.
Positive-definite $p \times p$ matrices are denoted by Greek 
letters $\Alpha,\Beta,\gam$ and matrix-valued functions by calligraphic capital letters such as $\blC(s)$ or $\blF(s)$. 
Last, for $\Alpha, \Beta \in \mathbb{R}^{p\times p}$ and $\Beta$ is nonsingular,
{ we use the notation $\frac{\Alpha}{\Beta}:= \Alpha \Beta^{-1}$ (right inversion).}
The matrix $E_k \in \mathbb{R}^{mp\times p}$ has zero elements except for the
$p\times p$ identity matrix in the $k$-th block, $E_k =[0,\ldots, 0,I, \ldots, 0]^T$.

\section{Block Gau{\ss}ian Quadratures, Stieltjes strings and continued fractions}\label{sec:BlLanc}

 {\begin{center}
\begin{minipage}{.65\linewidth}
 \begin{algorithm}[H]
\caption{Block Lanczos iteration}\label{alg:blockLanc}
\begin{algorithmic}
\normalsize
\State Given $m$, $A\in{\mathbb R}^{n\times n}$ s.p.d., 
$B\in{\mathbb R}^{n\times p}$ with orthonormal columns 
\State $Q_1 =B$
\State $W = AQ_1$
\State $\Alpha_1 = Q_1^T W$
\State $W = W - Q_1 \Alpha_1$
\For{$i= 2,\dots, m$} 
 	\State $Q_i\Beta_{i}=W$ \hskip 0.5in [QR decomposition of $W$]
	\State $W = AQ_i- Q_{i-1}\Beta_{i}^T$
 	\State $\Alpha_i = Q_i^T W$
 	\State $W = W - Q_i \Alpha_i$
\EndFor 
\end{algorithmic}
 \end{algorithm}
\end{minipage}
\end{center}}
\vspace{0.5cm}

Let us assume that we can perform $m$ steps, with $mp \le n$, of the block Lanczos iteration (Algoritm~\ref{alg:blockLanc}) without breakdowns 
or deflation \cite{O'Leary1980,GOLUB1977}. As a result, 
the orthonormal block vectors $Q_i\in \mathbb{R}^{n \times p}$ form 
the matrix ${\boldsymbol Q}_m=[Q_1, \ldots, Q_m]\in \mathbb{R}^{n\times mp}$, whose columns
contain an orthonormal basis for the block Krylov subspace
\[\mathcal{K}_m(A,B)={\rm blkspan} \{B, AB, A^2 B , \dots A^{m-1} B \}.\]
The Lanczos iteration can then be compactly written as
\be\label{eq:LancRel}
A{\boldsymbol Q}_m={\boldsymbol Q}_m T_m + Q_{m+1} \Beta_{m+1}E_m^T, 
\ee
where $T_m$ is the symmetric positive definite block tridiagonal matrix 
\be\label{eq:T}
T_m=
\begin{pmatrix}
\Alpha_1 	& \Beta_2^T 	& {}		&{}			& {}& {}& {}\\
\Beta_2 	& \Alpha_2	& \Beta_3^T	&{}			& {}& {}& {}\\
{}			& \ddots 	& \ddots 	& \ddots 	& {}& {}& {}\\
{}			& {}		& \Beta_{i}& \Alpha_i & \Beta_{i+1}^T & {}& {}\\
{}			& {}		&{}			& \ddots 	& \ddots 	& \ddots& {}\\
{}			& {}		&{}			& {}	& \Beta_{m-1}& \Alpha_{m-1}& \Beta_m^T\\
{}			& {}		&{}			& {} 	& {} 	& \Beta_m& \Alpha_m\\
\end{pmatrix} ,
\ee
and $\Alpha_i, \Beta_i \in\mathbb{R}^{p \times p}$ are the block coefficients
in Algorithm~\ref{alg:blockLanc}. 

Using the Lanczos decomposition, $\blF(s)$ can be approximated as
\be\label{eq:blapprox}
 \blF(s)\approx \blF_m(s)= E_1^T(T_m+sI)^{-1}E_1.
\ee
This approximation is known as a block Gau{\ss} quadrature rule.

The simple Lanczos algorithm using only three-term recursions without re-orthogonalization is known to be unstable due to computer round-offs. The instability is manifested by the loss of orthogonality of the Lanczos vectors and the appearance of spurious copies of the Lanczos eigenvalues. However, this instability only slows down the convergence speed to some degree and does not affect the accuracy of the final converged result \cite{Knizhnerman1996TheSL}.

To obtain a representation as a block-extension of Stieltjes strings we represent the $T_m$ (see equation \eqref{eq:T}) in a block $LDL^T$ decomposition
\begin{equation}\label{eq:defT}
T_{m} :=(\widehat{\boldsymbol K}_{m}^{-1})^T {J}_{m} \boldsymbol\Gamma^{-1}_m {J}_{m}^T \widehat{\boldsymbol K}_{m}^{-1}
\end{equation} 
where $\hKappa_{i}\in\RR^{p\times p}$, $\hKappa_1=I_p$ and $\gam_{i}\in\RR^{p\times p}$ all full rank, and
\[
{J}_{m}^T = 
\begin{bmatrix}
I_p & -I_p & ~ & ~ \\
~ & \ddots& \ddots & ~ \\
 ~& ~& \ddots & -I_p \\
 ~ & ~&~& I_p 
\end{bmatrix}\in\RR^{pm\times pm},
	\begin{array}{ll}
\widehat{\boldsymbol K}_{m}&={\rm blkdiag}(\hKappa_{1},\dots,\hKappa_{m})\\
{\boldsymbol \Gamma}_m&={\rm blkdiag}(\gam_{1},\dots,\gam_{m}) ,
\end{array}
\]
and
$\Alpha_1=(\hKappa_{1}^{-1})^T\gam_1^{-1}\hKappa_{1}^{-1}=\gam_1^{-1}$, 
$\Alpha_i=(\hKappa_{i}^{-1})^T(\gam_{i-1}^{-1}+\gam_{i}^{-1})\hKappa_{i}^{-1}$ and 
$\Beta_i= { -(\hKappa_{i}^{-1})^T \gam_{i-1}^{-1} \hKappa_{i-1}^{-1}}$ for $i=2,\ldots,m$. 

The matrices $\gam_i$ and $\hKappa_{i}$ can be computed directly during the block-Lanczos recursion using the coefficients $\Alpha_i$'s and $\Beta_i>0$'s. 
 
This is reported in Algorithm~\ref{alg:ExtractGam}, given in Appendix~A, and previously derived for a different parametrization in \cite{ZaslavskySfraction}.

We first convert $T_m$ to pencil form using the factors introduced in \eqref{eq:defT}. To this end, we introduce the matrices $\hgam_j$
\be\label{eq:decompose}
\hgam_j =\hKappa_j^T \hKappa_j, \quad j=1,\ldots,m.
\ee
The matrices $\gam_j$ and $\hgam_j$ are known as the Stieltjes parameters, and they are both s.p.d. if block Lanczos runs without breakdown.

Let $ Z_m:= {J}_m \Gamma_m^{-1} {J}_{m}^T$ and $\widehat{\boldsymbol 
 \Gamma}_m={\rm blkdiag}(\hgam_{1},\dots,\hgam_{m})$.
Then, due to the initial condition $\hgam_1=I_p$, the function $\blF_m$ approximating ${\cal F}=B^T (A+sI)^{-1}B$ can be rewritten using the pencil form $(Z_m,\widehat{\boldsymbol \Gamma}_m)$, that is
\be\label{eq:Gauss }
\blF_m(s)=E_1^T(T_m+sI)^{-1}E_1= E_1^T( Z_m+s \widehat{\boldsymbol 
 \Gamma}_m)^{-1}E_1.
\ee

Thus 
$\blF_m(s)$ corresponds to the first $p\times p$
block of the solution ${ U}_m$ of the linear system
\be\label{eq:LAform}
( Z_m+s \widehat{\boldsymbol \Gamma}_m){ U}_m =E_1,
\ee
with block column vector ${ U}_m = [\blU_1; \ldots; \blU_m]$, with $\blU_i$ a $p\times p$ block.

Appending ${ U}_m $ by ``boundary" blocks $\blU_0,\,\blU_{m+1}\in\CC^{p\times p}$ and introducing a fictitious $\gam_0=I_p$ we rewrite \eqref{eq:LAform} as a finite-difference Stieltjes string
\begin{eqnarray}
\frac 1 \gam_0\left(\blU_1 - \blU_0\right) &=&-I_p \label{eqn:line1}\\
\frac 1 {\gam_{i-1}}\left(\blU_{i} -\blU_{i-1}\right) 
 - \frac 1 {\gam_{i}} \left(\blU_{i+1}-\blU_{i}\right)+s\hat\gam_i \blU_i
 &=&0, \quad
i=1,\ldots, m\label{eq:linei}\\
\blU_{m +1}&=& 0. \label{eqn:linem}
\end{eqnarray}
where the first line can be interpreted as a block Neumann condition and the last as a block Dirichlet condition.

For $p=1$ Kreĭn \cite{Krein1952} interpreted the system (\ref{eqn:line1}-\ref{eqn:linem}) as a so-called Stieltjes string
\be\label{eq:ss} -u_{xx}+s\mathbf{m}_m(x)u=0 \ee on a positive interval $[0,x_{m+1}]$, with boundary conditions $u_x(0)=-1$, $u(x_{m+1})=0$, and discrete mass distribution
\be\label{eq:mass} \mathbf{m}_m(x)=\sum_{i=1}^m\hgam_i\delta(x-x_i),\ee
where $x_1=0$, $x_{i+1}=x_i+\gam_i$, $i=1,\ldots, m$,
with $\blU_i=u(x_i)$. See \cite{Zimmerling_KreinNudel} for a more detailed derivation.
 The discrete Sieltjes string interpretation is fundamental for the derivation of our algorithm.\footnote{The Stieltjes parameters $\gam_{i}$ and $\hgam_i$ can be viewed as the primary and dual steps in a more detailed finite-difference interpretation \cite{druskin1999Gaussian} not used here.}
 
The eigendecomposition of $T_m$ allows an interpretation of $\blF_m(s)$ as block Gau\ss ian quadrature \cite{MeurantBook}. Here, we interpret it as the transfer function or Neumann-to-Dirichlet map of the matrix Stieltjes string. 
\section{Block Kreĭn-Nudelman extension}\label{sec:Krein-Nudelman1}

 { Extending this analogy, for problems arising from differential equations on a semi-infinite interval ($\RR_+$), Kreĭn and Nudelman introduced to the Stieltjes string a finite-difference variant of the Sommerfeld absorbing condition. We write such a condition for $p\ge 1$ in the form
\be\label{eq:sommerfeld}
({\sqrt{s}}{\phi} +\varphi ) \blU_{m+1}= - \frac 1 {\gam_{m}} (\blU_{m+1}-\blU_{m}),
\ee
replacing \eqref{eqn:linem},
where $\phi,\varphi \in\RR^{p\times p}$ are some s.p.d. matrix parameters whose choice we discuss later.  
Returning to a slightly modified original interpretation of \cite{KreinNudelman1989} for $p=1$, let us extend \eqref{eq:ss} from $[0,x_{m+1}]$ to a problem $[0,\infty)$ by an equation for the first spherical harmonic of the Helmholtz equation on $k$-dimensional sphere
\be\label{eq:ssKN} 
-(\sigma(x) u_x)_x+s\hat {\mathbf m}_m u=0, 
\ee
with the same boundary condition $u(0)_x=-1$ infinity condition $u(\infty)=0$ for $s\notin \RR_+$,
here $\sigma(x)= 1$ for $x\in[0,x_{m+1}]$ and $\sigma(x)=(x/x_{m+1})^{k-1}$ otherwise. Further we replace $\mathbf{m}_m$ given by \eqref{eq:mass} 
with 
\be\label{eq:masssinf} \hat {\mathbf{m}}_m=\mathbf{m}_m+(x/x_{m+1})^{k-1}\phi^{2}\eta(x-x_{m+1}),\ee
where  $\eta(x-x_{m+1})$ is the Heaviside step function.
Parametrization \eqref{eq:masssinf} yields outgoing waves solutions \eqref{eq:ssKN} on $[x_{m+1},\infty]$
 $u(x)=\blU_{m+1} \frac{e^{-\sqrt{s}{\phi} x}}{x^{\frac{k-1}{2}}}$.
Then condition 
\be\label{eq:sommerfeldE} ({\sqrt{s}}{\phi}+\varphi) u|_{x=x_{m+1}}=-u_x|_{x=x_{m+1}}\ee
truncates domain $[0,\infty]$ to $[0,x_{m+1} ]$, which for $\lim_{x\to\infty}\varphi =0$ is equivalent to the Sommerfeld radiation condition for outgoing waves and \eqref{eq:sommerfeld} is a finite-difference approximation to \eqref{eq:sommerfeldE} with a one sided derivative approximation.
Then for $k=1$ we find $\varphi=0$ and in general $\varphi=\frac{(k-1)}{2 x_{m+1}}$, satisfying the radiation condition.


We denote
\[\hat\blF_m^{\phi,\varphi}(s)= \blU_1\]
corresponding to Eqs.~(\ref{eqn:line1},\ref{eqn:linem},\ref{eq:sommerfeld}) and call $\hat \blF_m^{\phi,\varphi}(s)$ the Kreĭn-Nudelman quadrature. 

Thus, returning to the general case $p\ge 1$ and
using \eqref{eq:sommerfeld} we eliminate $\blU_{m+1}$ from the $m$-th equation of
\eqref{eq:linei} 
and obtain $$\frac{1} {\gam_{m-1}}\left(\blU_{m} -\blU_{m-1}\right) 
 - \frac 1 {\gam_{m}} \left(\left[\frac {1} {\gam_{m}}+\varphi+{\sqrt{s}}{\phi}\right]^{-1}\frac {1} {\gam_{m}}-I\right)\blU_{m}+s\hat\gam_m \blU_m=0.
 $$
 Combining this with the remaining equations of \eqref{eq:linei} and symmetrizing the obtained matrix pencil, we obtain
 \be\label{eq:KreinNudelman} 
 \hat\blF_m^{\phi,\varphi}(s)=E_1^T(\hat T_m^{\phi,\varphi}(s)+sI)^{-1}E_1,
 \ee
 where $\hat T_m^{\phi,\varphi}(s)$ coincides with $T_m$ except last diagonal element
\be\label{eq:hatal}\hat \Alpha_m^{\phi,\varphi}(s)=\Alpha_m-(\hKappa_{m})^{-T}\gamma_m^{-1}(\gamma_m^{-1}+\varphi+\sqrt{s}\phi)^{-1}\gamma_m^{-1}( \hKappa_{m})^{-1}.\ee
By construction, \eqref{eq:hatal} yields a symmetric $\hat \Alpha_m$ and  $ \lim_{\phi\to\infty} \hat\blF_m(s)^{\phi,\varphi}(s)=\blF_m(s)$. The other limiting   corresponds to the Gau\ss -Radau quadrature $\blTF_{m}(s)$ given by $ \lim_{\phi,\varphi\to 0} \hat\blF_m^{\phi,\varphi}(s)=\blTF_m(s)$ as defined in \cite{zimmerling2025monotonicity}.
There, it was shown that $\forall s\in\RR_+$
 \be\label{bound:GR} \ldots \blF_{m-1}(s)<\blF_{m}(s)\le \blF(s)\le\blTF_{m}(s)< \blTF_{m-1}\ldots, \ee.
The following proposition extends this bound to the Kreĭn-Nudelman quadrature and shows, that, similar to Gau{\ss} and Gau\ss -Radau quadrature, it is a Stieltjes function; however, unlike those, it is not meromorphic. 
 \begin{proposition}\label{prop:main}
 For $0< \phi,\varphi <\infty $ 
 \begin{enumerate}
 \item $\hat \blF^{\phi,\varphi}_m(s)$ is a Stieltjes function with the branch-cut on $\RR_-$;
 \item $\forall s\in\RR_+$
 \be\label{in:KN} \blF_{m}(s)\le \hat \blF^{\phi,\varphi}_m(s)\le\blTF_{m}(s).\ee
 \end{enumerate}
 
 \end{proposition}
 The proof is given in Appendix~\ref{ap:ext}.
 Proposition~\ref{prop:main} states the exact solution, and the Kreĭn-Nudelman approximation satisfies the same two-sided bounds, which indicates potential accuracy improvement in the latter. In the following section, we consider the optimal choice of its parameters and quantify accuracy improvement. 




\section{Spectral optimization in Hermite-Pad\'e framework}\label{sec:Hermite-Pade}
{\subsection{Reflection optimization}

Now, it remains to choose the s.p.d. matrix parameters $\phi$ and $\varphi$ to minimize the reflection due to the truncation of the Lanczos recursion. 
Here, we propose two approaches.
\subsubsection{Matchting Gau{\ss} and Gau\ss -Radau estimates} The Gau{\ss} quadrature $\blF_m$ can be equivalently represented as 
$\blF_m(s)=u|_{x=0}$ where $u$ satisfies oddly extended problem from interval $[0,x_m]$ to interval $[0,2x_m]$ 
for equation 
\be\label{eq:Kreinext} -u_{xx}+\hat{m}_m(x)=0, \ee
\be\label{eq:bcG} u_x(0)=-1, \qquad u_x(2x_m)=1\ee
with $\hat{m}_m(x)=\mathbf{m}_m(x)$ for $x\in[0,x_m]$ and $\hat{m}_m(x)=\mathbf{m}_m(2x_m-x)$ for $x\in[x_m,2x_m]$.
Likewise the Gau\ss -Radau quadrature will be {\it approximately} defined as the even extension $\tilde \blF_m(s)\approx u|_{x=0}$ for the problem \eqref{eq:bcG} with condition
\be\label{eq:bcGR} u_x(0)=-1, \qquad u_x(2x_m)=-1.\ee The exact condition would be extension on $[0,2x_{m+1}]$ however we assume that $\gamma_m/x_{m}$ will be small for large $m$, so the difference will be insufficient. It was shown in \cite{zimmerling2025monotonicity} that average of the Gau{\ss} and Gau\ss -Radau quadratures
$\frac{\tilde \blF_m(s)+ \blF_m(s)}{2}$ is a more accurate approximation of $\blF_m(s)$ compared to $\tilde \blF_m(s)$ and $ \blF_m(s)$ taken separately. By superposition,
we obtain $\frac{\tilde \blF_m(s)+ \blF_m(s)}{2}=u|_{x=0}$, where $u$ satisfies equation \eqref{eq:Kreinext} with boundary conditions \be\label{eq:bcGRG} u_x(0)=-1, \qquad u_x(2x_m)=0.\ee Thus, if we use this wave analogy, the averaging removes the first reflection resulting from the truncation of the Lanczos recursion. So effectively, we have the twice as long recursion, which extension indeed not necessary coincides with the network produced with the true Lanczos recursion of length $2m$. Even if they coincide, this recursion has reflections from the truncation after $2m$ steps. For $s\to \infty$, the latter reflection will be small compared to the dominant part of the error of the Gau\ss ian quadrature, which is the reflection resulted from truncation after $m$ steps. That gives a condition for optimal $\phi$ and $\varphi$, i.e., it should satisfy
\be\label{eq:match}\lim_{s\to\infty}\frac{0.5[\tilde \blF_m(s)+ \blF_m(s)]-\blF_m(s)}{ \hat\blF_m^{\phi,\varphi}(s)- \blF_m(s)}=\frac{0.5[\tilde \blF_m(s)-\blF_m(s)]}{ \hat\blF^{\phi,\varphi}_m(s)- \blF_m(s)}=1.\ee
In practice, this means that we minimize the mismatch between $\hat\blF^{\phi,\varphi}_m(s)$ and $0.5[\tilde \blF_m(s)+ \blF_m(s)]$ for large values of $s$, where the Gau\ss -Radau error bound $||\tilde \blF_m(s)- \blF_m(s) ||(||\blF_m(s)||)^{-1} $ is small.

\subsubsection{Targeted smoothing of spectral measure}
A disadvantage of the above matching approach is that, by construction, it is similar to averaged Gau{\ss} and Gau\ss -Radau quadrature at the matching area. Thus, in our experiments using this matching condition, the advantage of $\hat\blF_{m}^{\phi,\varphi}$ over averages Gau\ss -Radau diminishes during convergence since $\phi,\varphi$ are obtain by matching them for shifts $s$ where convergence occurred. Here, we propose a qualitative approach by maximizing the smoothness of the spectral measure in the regions of its poor convergence.

It is known that both $\blF_m(s)$ and $\tilde \blF_{(m+1)}(s)$ are Pad\'e approximations of $\blF(s)$ at $s=\infty$ \cite{zimmerling2025monotonicity}. Here, however, we consider a class of problems, when $\blF(s)$ is an approximation of a transfer function with continuous spectrum, i.e., one with branch cut(s) along $\RR_-$. 
It is known that Hermite-Pad\'e are superior to simple Pad\'e for problems where the branch cut is explicitly introduced to the approximation \cite{Aptekarev,Suetin}. For instance, the square root function, as in our case where we approximate a transfer function with continuous spectrum by Hermite-Pad\'e function also with continuous spectrum.  
In addition to the same moment matching as in $\tilde \blF_{(m+1)}(s)$ we require that the poorly convergent poles of $\hat \blF_{m}^{\phi,\varphi }(s) $ (also known as the scattering poles or resonances in wave problems) to be located as far as possible from the physical Riemann sheet. We will achieve this by minimizing the $L_2$ norm of  $\hat \blF_{m}^{\phi,\varphi }(s)  $ along the negative real semiaxis, which is the supposed branch-cut of the underlying problem.

We should keep in mind, however, that the branch-cut assumption may not hold, and even if it does,  $A$ is not an operator with a continuous spectrum but still a finite-dimensional (even possibly very large-scale) approximation. As such, it may still have well-separated eigenvalues, corresponding to the part of the spectrum (normally, upper for the PDE problems) where such discretization loses accuracy.  Such eigenvalues, however,  are well-approximated by the Gau\ss ian quadratures. This reasoning allows us to introduce a spectral weight, given by an error estimate for the nodes of Gau\ss ian quadrature.  For that, we will use a well-known formula for  the residual of the Krylov subspace projection that can be obtained by multiplication \eqref{eq:LancRel} by $(T_m+sI)^{-1}E_1$ from the right as 
\[{( A+sI)}{\boldsymbol  Q}_m(T_m+sI)^{-1}E_1-B=
 Q_{m+1} \Beta_{m+1}E_m^T(T_m+sI)^{-1}E_1,\]
 from which we obtain an easily computed weight function as the norm of the residual relative to the norm of the Krylov  solution
 \be\label{eq:weight} \omega(s)=\frac{\|{( A+sI)}{\boldsymbol  Q}_m(T_m+sI)^{-1}E_1-B\|}{\|{\boldsymbol  Q}_m(T_m+sI)^{-1}E_1\|}=\frac{\|\Beta_{m+1}E_m^T(T_m+sI)^{-1}E_1\|}{\|(T_m+sI)^{-1}E_1\| }.
 \ee

Thus, we compute $\phi$, $\varphi$ via minimization 
\begin{align} \label{cond:opta} 
\phi,\varphi &= \argmin_{\varphi\ge0,\phi>0 }\Omega_{\phi,\varphi, \epsilon},  \\
\Omega_{\phi,\varphi,\epsilon}&:=\int_{R_-+i\epsilon}  \omega^2(s) \|
 {E_{1}}^T(\hat T_m^{\phi,\varphi}(s)+sI)^{-1}E_1 \|_F^2 ds \nonumber
 \end{align} for some small $\epsilon>0$, $\|.\|_F$ is the Frobenius norm.

\subsection{Spectral bound}
The poles of $\blF_m,\tilde\blF_{m+1}$ lie on the negative semiaxis, so both the optimization algorithms should deliver an approximation different from the Gau{\ss} and Gau\ss -Radau quadratures in this case.
Here, we quantify the acceleration brought by the introduction of the Hermite-Pad\'e term in our approximant. 
For simplicity, we consider here the SISO case with $p=1$, and analyze this problem in a more general setting than \eqref{eq:prob1}, i.e., replace the resolvent with a rather general matrix function $f$. 

Let $\rho(x)$ and $\hat \rho_m(x)$ be spectral densities of respectively $\blF(s)$ and $\blF_m(s)$, $x\in\RR_-$, i.e., 
\[\rho(-x)=\frac{1}{\pi}\lim_{\epsilon \to 0}\Im \blF(x+i\epsilon ), \quad \hat \rho_m(-x)=\frac{1}{\pi}\lim_{\epsilon \to 0}\Im  \hat\blF_m^{\phi,\varphi}(x+i\epsilon ) .\]
 Due using spectral theorem, any  function $f(x)$ continuous on $\RR_+$ defines 
\be\label{eq:spec} B^T f(A)B=\int_0^\infty f(x)\rho(x)dx, \quad E_1^Tf(  \hat T^{\phi,\varphi}_{m+1})E_1=\int_0^\infty f(x)\hat \rho_m(x)dx.\ee

Here, we extend the bounds of \cite{druskin1989two,Knizhnerman1996TheSL} from Lanczos decomposition and Gau\ss ian quadrature to the Kreĭn-Nudelman extension. 
\begin{lemma}\label{lem:1}
Let $ \sum_{i=1}^ {2m} \bbF_i s^{-i} $ and $ \sum_{i=1}^ {2m} \hat \bbF_{m , i} s^{-i}$ be a formal Laurent expressions of $\blF$ and $ \blF_{m }$.
Then
{\be\label{EQ:Laurent}\bbF_i =B^T(-A)^{i-1}B=\hat \bbF_{m {, i}}= E_1^T (- T_m)^{i-1} E_1,\quad i=1,\ldots,2m.\ee}
\end{lemma}
\begin{proof}
Expanding $s^{-1}(\frac{A}{s}+I)^{-1}$ with respect to the powers of $\frac{A}{s}$,
we obtain $\bbF_i = B^T(-A)^iB$. Likewise, expanding $s^{-1}(\frac{\hat T_{m+1}^{\phi,\varphi}(s)}{s}+I)^{-1}$ with respect to powers of $\frac{\hat T_{m+1}^{\phi,\varphi}(s)}{s}$ we obtain\footnote{In abuse of notation we write $\hat T_{m+1}^{\phi,\varphi}(s) E_1$, in the understanding that $E_1$ is an $(m+1)p\times p$ canonical block-identity vector matching the dimensions of $\hat T_{m+1}$. } 
\be\nonumber
\hat \bbF^{\phi,\varphi}_{m+1  ,i}= E_1^T (\hat T_{m+1}^{\phi,\varphi}(s))^i E_1=(\hat T_{m+1}^{\phi,\varphi}(s))^{i/2}E_1)^T (\hat T_{m+1}^{\phi,\varphi}(s))^{i/2}E_1 
\ee
if $i$ is even and otherwise 
\be\nonumber
\bbF^{\phi,\varphi}_{m+1,i}= E_1^T (\hat T_{m+1}^{\phi,\varphi}(s))^i E_1=(\hat T_{m+1}^{\phi,\varphi}(s))^{(i-1)/2}E_1)^T (\hat T_{m+1}^{\phi,\varphi}(s))^{(i+1)/2)}E_1.
\ee
Due to the tridiagonal structure of  $\hat T_{m+1}^{\phi,\varphi}(s)$, the first component of $\hat T_{m+1}^{\phi,\varphi}(s)^{i}E_1$ coincides with $T_m^iE_1$ for $i\le m$, so we can replace $\hat T^{\phi,\varphi}_{m+1}(s)$ with $T_m$.

This allows us to apply Lemma~2 of \cite{Knizhnerman1996TheSL} for $\epsilon=0$ and obtain $\bbF_i =\hat \bbF^{\phi,\varphi}_{m+1 ,i}$.
\end{proof}

\begin{proposition}\label{prop:2}
Let $\|A\|\le d$, $f(x)$ be continuous on $[0,d]$.
Then
\be\label{eq:bond2} 
|B^T f(A)B-E_1^Tf(\hat T_m^{\phi,\varphi}(s))E_1|<\min_{\deg p\le 2m-1} \max_{x\in[0,d]}|f-p_{2m-1}|\int_{0}^{d}|\rho(x)-\hat\rho_m(x) | dx .\ee 
\end{proposition}

\begin{proof}
Expanding $(s+x)^{-1}$ using a Laurent series with respect to $s$  
and substituting into \eqref{eq:spec} we obtain the Stieljes moments
\[\bbF_i=\int_0^\infty x^i\rho(x)dx, \quad \hat \bbF_{m , i}=\int_0^\infty x^i\hat \rho_m(x)dx.\]
Then, combining the above identities with Lemma\ref{lem:1} we obtain
\begin{eqnarray*}
 \label{eq:ident}B^Tf(A)B-E_1^Tf(\hat T^{\phi,\varphi}_m)E_1=\\ B^Tf_m(A)B-B^Tp_{2m-1}(A)B -[E_1^Tf(T_m)E_1-E_1^Tp_{2m-1}(T_m)E_1]=\\ \int_0^d [f(x)-p_{2m-1}(x)][\rho(x)-\hat \rho_m(x)] .\end{eqnarray*}
 Applying H\"older inequality to r.h.s. of the last identity, we obtain 
 \eqref{eq:bond2}.
\end{proof}

For $\beta_{m+1}=0$ $\rho=\hat\rho_m$ and the solution will be exact. For $\beta_{m+1}\ne 0$ this proposition will remain valid for $a=b=0$, i.e., for the Gau{\ss} quadrature, in which case $\int_{0}^{d}|\rho(x)-\hat\rho_m(x) | dx=2$, and the bound becomes similar to \cite{druskin1989two}[Theorem~2], only with $m-1$ instead $2m-1$ due to the doubling of the polynomial moments in the quadratures. 

Let us now assume a continuous spectral measure $\rho$. By choosing a spd bounded $\phi,\varphi$ satisfying \eqref{eq:match} producing maximal ``absorption" it moves poorly converged poles of the spectral measure from the branch cut to the second Riemann sheet, thus $\hat \rho_m(x)$ becomes analytic in a neighborhood $[0,d]$ on both the Riemann sheets. If $\rho(x)$ is similarly analytic, e.g., as in scattering problems in the unbounded domains \cite{Zworski},
then as the Hermite-Pad\'e approximant $\hat \rho_m(x)$ also converges to $\rho(x)$ \cite{Aptekarev,Suetin}. For continuous $\rho(x)$ the Chebyshev bound $\min_{\deg p\le 2m-1} \max_{x\in[0,d]}|f-p_{2m-1}|$ gives the best asymptotic convergence rate, so the factor \[\int_{0}^{d}|\rho(x)-\hat\rho_m(x) | dx\] gives convergence acceleration due to implicitly using the Hermite-Pad\'e approximation in the Kreĭn-Nudelman approach. In our optimization, we target the estimate of $\int_{0}^{d}|\rho(x)-\hat\rho_m(x) | dx$ via the residual of the Stieltjes string.} This factor decays exponentially if the domains of analyticity do not collapse. In comparison, the bound for the reduction of errors in the averaged Gau{\ss} and Gau\ss -Radau quadratures gives only a constant error reduction factor \cite{zimmerling2025monotonicity}. 


\section{Numerical Examples}\label{sec:NumEx}
In this preprint, we limit ourselves to 2D diffusion and 3D electromagnetism. We find $\phi,\varphi$ via minimization of \eqref{cond:opta} using the Nelder--Mead Simplex Method~\cite{LagariasOpt}. We optimize this expression in an interval $s\in[-d,\, 0]$, rather than over $\mathbb{R}_-$, where $s=-d$ corresponds to about the first 10\% of the dense spectrum, i.e. 10\% of the interval of support of $\omega(s)$.

\subsection{2D Diffusion}

We discretize the operator  
\[
\sigma({\bf x})^{-\frac{1}{2}}\Delta\sigma({\bf x})^{-\frac{1}{2}}
\]
on an unbounded domain using second-order finite differences over a $300\times 300$ grid. In the central region of the grid, where $\sigma(\bx)$ varies, we set the grid spacing to $\delta_x=1$.  

{ This formulation is relevant to both heat transfer and the scalar diffusion of electromagnetic fields. To approximate the problem numerically, we impose a Dirichlet boundary condition on a bounded domain, employing $N_{opt}=10$ exponentially increasing grid steps in the exterior while maintaining a uniform grid in the interior. By using the optimal geometric factor $\exp{(\pi/\sqrt{10})}$, as suggested in \cite{idkGrids}, we achieve spectral accuracy in the exterior region, effectively approximating the boundary condition at infinity through exponential decay.}  

The resulting grid and function $\sigma(\bx)$ are illustrated in Figure~\ref{fig:Heat2config}. The figure only depicts the first two steps of the exponentially increasing grid, as subsequent steps grow rapidly. Here, $B$ is represented as a vector of discrete delta functions with nonzero entries at the grid nodes corresponding to the transducer locations, labeled as {\it Trx} in the figure. In the preliminary SISO results presented here we only consider the first transducer.

Figure~\ref{fig:HeatReal} presents the convergence results for $s=2.31\cdot 10^{-4}$, while Figure~\ref{fig:HeatImag} shows the convergence behavior for purely imaginary shifts $s=\imath2.31\cdot 10^{-4}$. In the latter case, block Lanczos exhibits slower convergence. The Kreĭn-Nudelman formulation converges faster than the Gau{\ss} and Gau\ss -Radau quadrature. The averaged Gau{\ss} and Gau\ss -Radau converge smoothly than the Kreĭn-Nudelman formulation, and future work will make the optimization routine used to find $\varphi$ and $\phi$ more robust since the oscillations observed in the convergence curve originate from a lack of robustness of the current optimization framework. That is, smoothing the recovered $\phi$ and $\varphi$ values over the iterations smoothens the convergence. Nonetheless, when fixing the iterations to say $m=400$ and examining the error for a range of real shifts, as shown in Figure~\ref{fig:HeatSweepReal}, and imaginary shifts, as shown in Figure~\ref{fig:HeatSweepImag} we see that the Kreĭn-Nudelman formulation consistently outperforms all other considered approaches.

\begin{figure}[h!]
	\centering
	\includegraphics[width = 0.65 \linewidth]{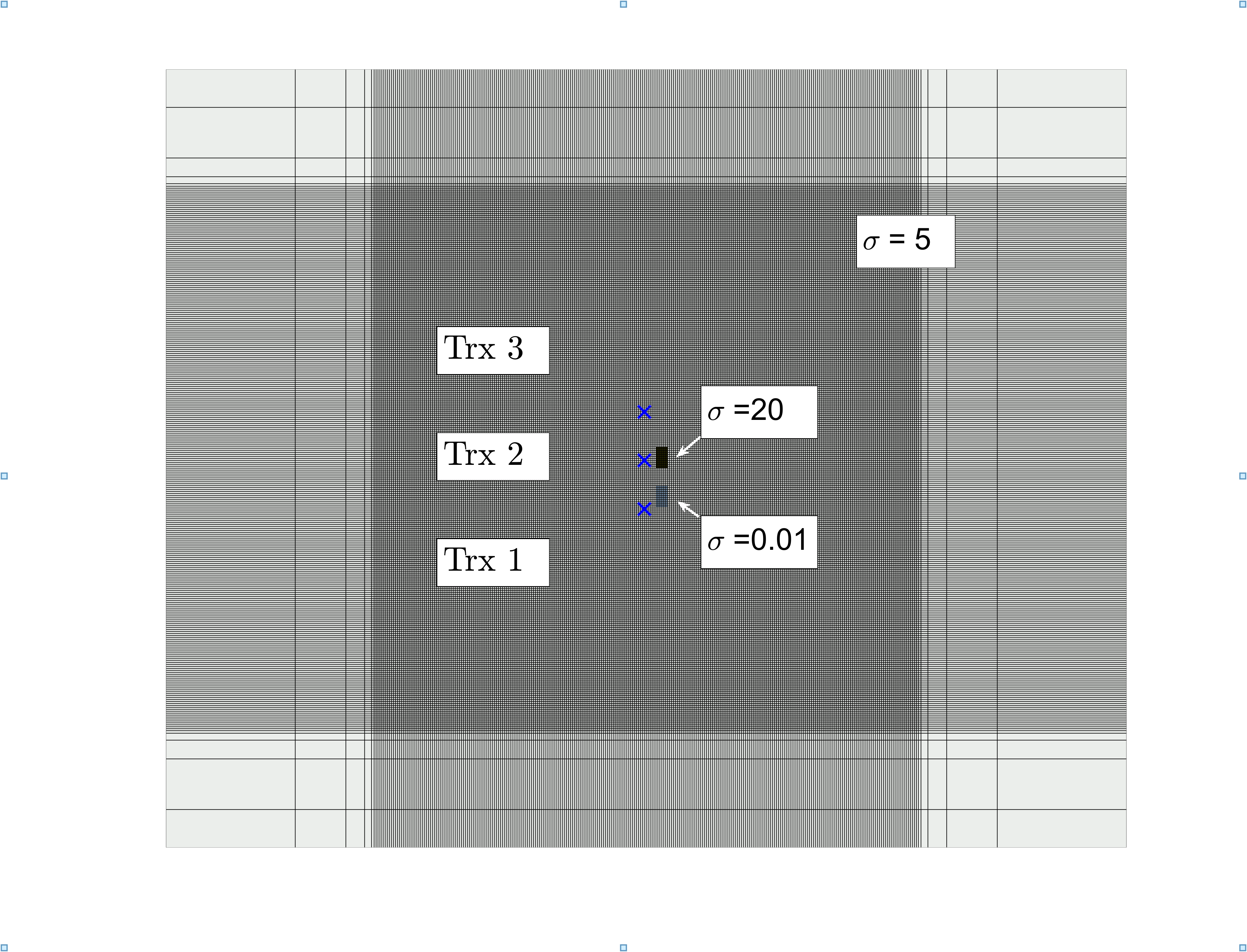}
	\caption{Grid, heat conductivity $\sigma(\bx)$ and transducer locations of the heat diffusion testcase.}\label{fig:Heat2config}
\end{figure}

\begin{figure}[h!]
 \centering
 
 \begin{subfigure}[b]{.47\linewidth}
 \centering
 \includegraphics[width = \linewidth]{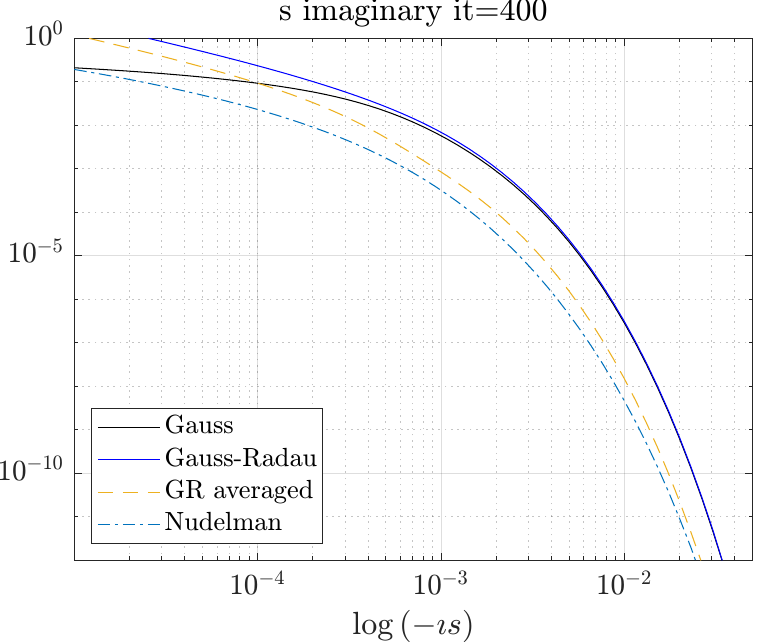}
 \caption{Error norms for various imaginary values of $s$ after $m=400$ iterations}
 \label{fig:HeatSweepImag}
 \end{subfigure}%
 \hspace{1em}
 \begin{subfigure}[b]{.47\linewidth}
 \centering
\includegraphics[width = \linewidth]{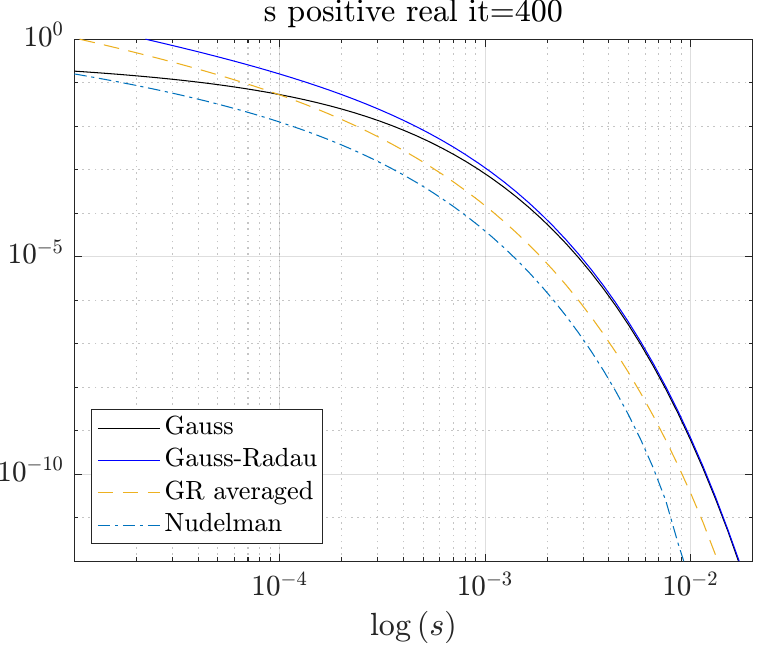}
 \caption{Error norms for various real values of $s$ after $m=400$ iterations.}
 \label{fig:HeatSweepReal}
 \end{subfigure}

 \begin{subfigure}[b]{.47\linewidth}
 \centering
 \includegraphics[width = 1\linewidth]{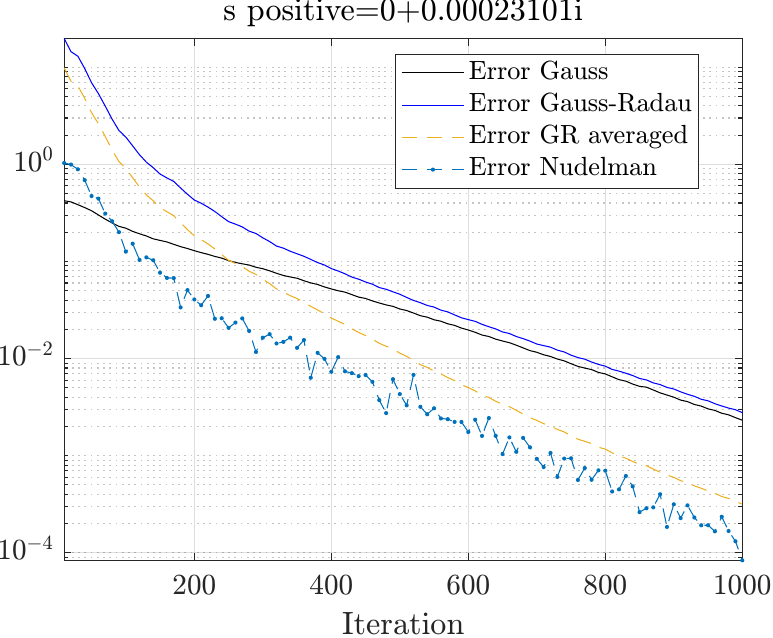}
 \caption{Convergence for an imaginary $s$ value.}
 \label{fig:HeatImag}
 \end{subfigure}%
 \hspace{1em}
 \begin{subfigure}[b]{.47\linewidth}
 \centering
\includegraphics[width = \linewidth]{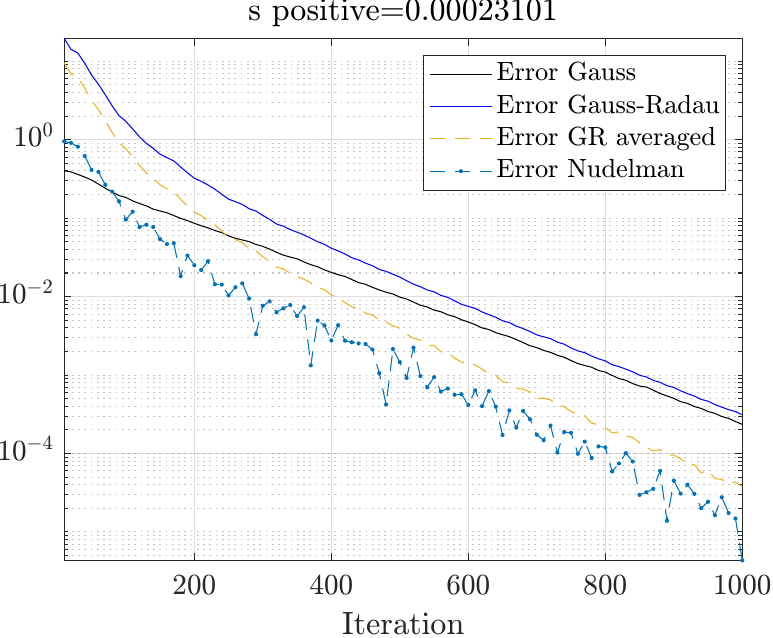}
 \caption{Convergence for a real $s$ value.}
 \label{fig:HeatReal}
 \end{subfigure}
\caption{SISO results for the heat equation.}\label{fig:HeatResults}
\end{figure}

\subsubsection{3D Electromagnetic diffusion}\label{sec:EMdiff} In this example, we consider the solution of the 3D Maxwell equations in the quasi-stationary (diffusive) approximation. Krylov subspace methods for approximating the action of matrix functions in this context were first introduced in \cite{druskin1988spectral} and remain a widely used approach in geophysical applications. The diffusive Maxwell equations in $\mathbb{R}^{3}$ can be rewritten as  
\be\label{eq:max}
(\nabla \times \nabla \times +\sigma(\bx) \mu_0 s)\mathcal{E}_{(r)}(\bx,s) = - { s}\mathcal{J}^{\rm ext}_{(r)}(\bx,s), \quad \bx\in \Omega,
\ee  
where $\sigma$ represents the electrical conductivity, $\mu_0$ is the vacuum permeability, and $\mathcal{J}^{\rm ext}_{(r)}$ is the external current density associated with the transmitter index $(r)$, generating the electric field $\mathcal{E}_{(r)}(\bx,s)$. For $s\notin\mathbb{R}_-$, the solution of \eqref{eq:max} vanishes at infinity,  
\be\label{eq:inftymax}  
\lim_{\bx\to\infty} \|\mathcal{E}(\bx)\|=0.  
\ee  

To approximate \eqref{eq:max}-\eqref{eq:inftymax}, we employ the conservative Yee grid, following the methodology outlined in \cite{druskin1988spectral}. Similar to the approach used for the diffusion problem in $\mathbb{R}^2$, we truncate the computational domain by imposing Dirichlet boundary conditions on the tangential components of $\mathcal{E}_{(r)}$ at the boundary. The grid is constructed with an optimally spaced geometric progression, ensuring spectral convergence of { condition~\eqref{eq:inftymax} at infinity}. In our setup, we use $N_{opt}=6$ geometrically increasing grid steps with a scaling factor of $\exp{(\pi/\sqrt{6})}$, as proposed in \cite{idkGrids}. The resulting grid consists of $N_x=80$, $N_y=100$, and $N_z=120$ points, yielding a matrix $A$ of size $N=2\,821\,100$.  

The conductivity $\sigma$ is homogeneous and set to $1$ throughout the domain, except for two inclusions where $\sigma=0.01$, as depicted in Figure~\ref{fig:4_config}. The excitation consists of six magnetic dipoles (approximated by electric loops): three positioned above the inclusions and three located at their center, as indicated by the arrows in Figure~\ref{fig:4_config}. This setup is analogous to tri-axial borehole tools used in electromagnetic exploration geophysics \cite{saputra2024adaptive}, leading to a matrix $B$ with six columns.  For the SISO (p=1) testcase, the upper $x$ oriented dipole is used.

Similar to the results for the heat operator, the 3D results show an advantage of the Kreĭn-Nudelman formulation over Gau\ss , Gau\ss -Radau, and averaged Gau\ss -Radau quadrature rules. In Figure~\ref{fig:EMReal}, the convergence for $s=0.061359$ is shown alongside the convergence for a purely imaginary  $s=\imath 0.0152$ in Figure~\ref{fig:EMImag}. Averaged Gau\ss -Radau and Kreĭn Nudelman develop an advantage in the area where ordinary Gau{\ss} quadrature converges linearly, in line with our theoretical results. After $m=600$ iterations, we see that this advantage is present across a range of imaginary shifts $\Im(s)>10{-2}$,  where we have entered linear convergence as shown in \ref{fig:EMSweep}. In conclusion, the averaged Gau\ss -Radau rule provides a lower error than the standard Gau\ss ian quadrature once the Gau\ss ian quadrature converges linearly. The new Kreĭn-Nudelman formulation enables lower approximation error across iterations and shifts for such cases. The convergence behavior for 3D PDEs is more regular in our experience than in the 2D case, and $\phi$ and $\varphi$ vary smoothly across iterations.




\begin{figure}[h!]
 \centering
 
 \begin{subfigure}[b]{.47\linewidth}
 \centering
 \includegraphics[width = 0.75\linewidth]{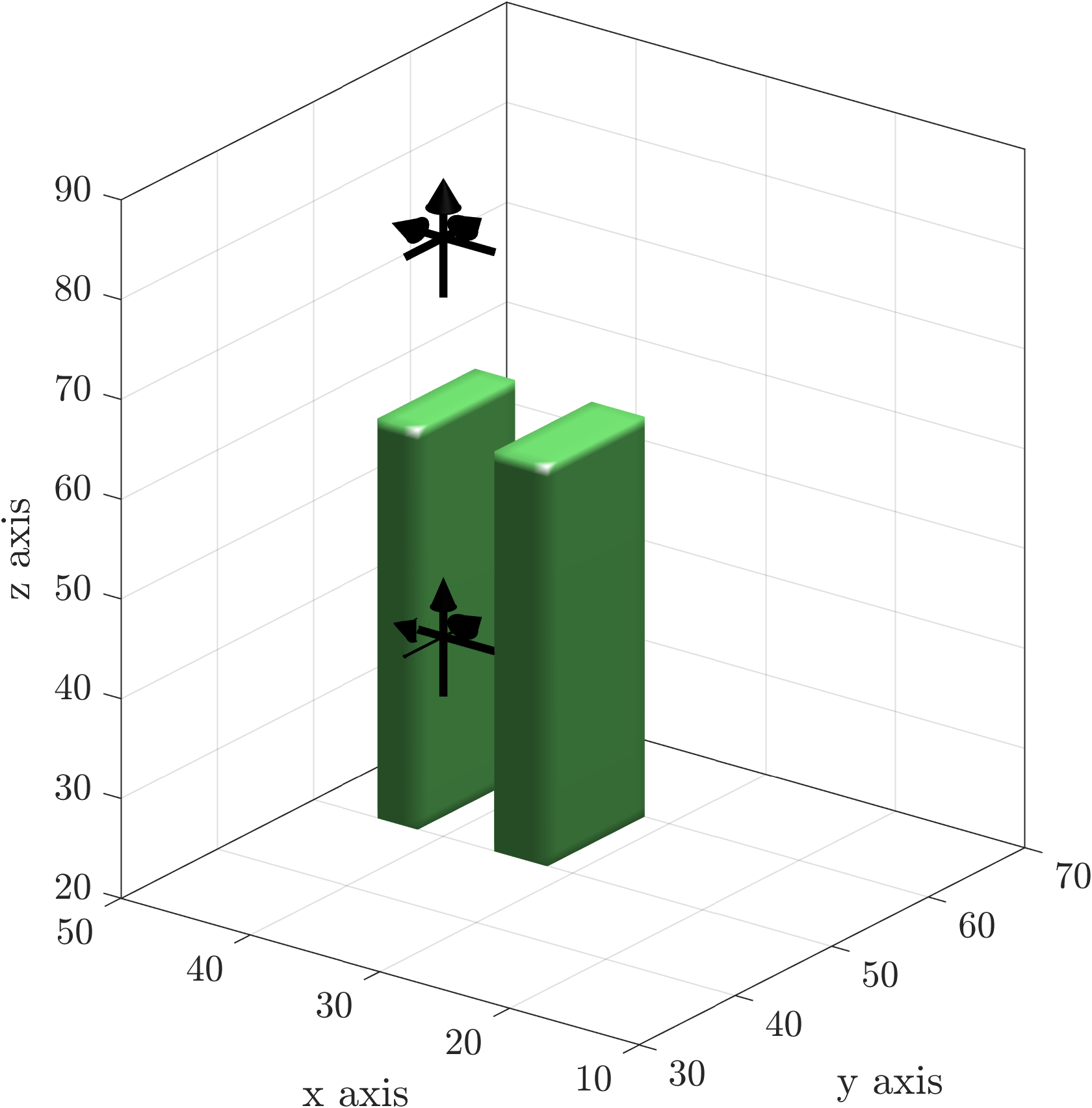}
 \caption{3D rendering of the simulated configuration with two inclusions.  For the SISO testcase, the upper $x$ oriented dipole is used.}
 \label{fig:4_config}
 \end{subfigure}%
 \hspace{1em}
 \begin{subfigure}[b]{.47\linewidth}
 \centering
\includegraphics[width = \linewidth]{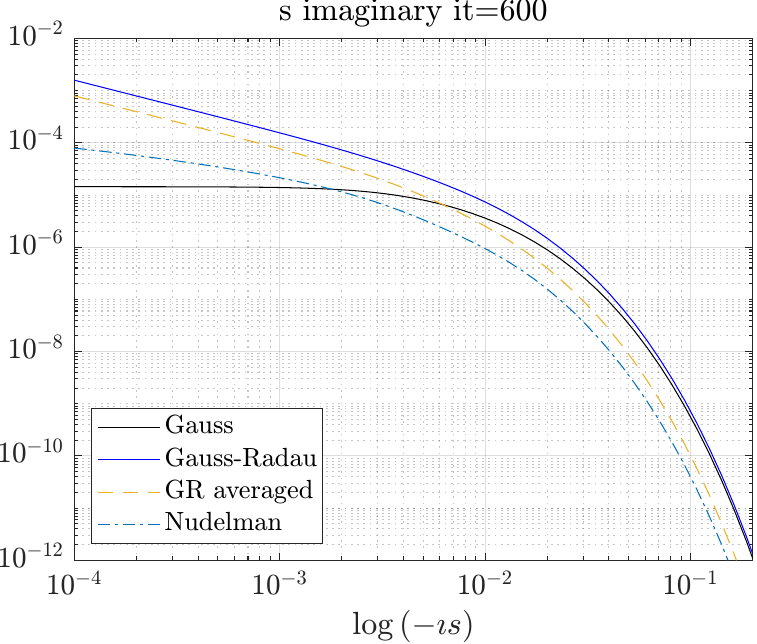}
 \caption{Error norms for various values of $s$ after $m=600$ iterations.}
 \label{fig:EMSweep}
 \end{subfigure}

 \begin{subfigure}[b]{.47\linewidth}
 \centering
 \includegraphics[width = 1\linewidth]{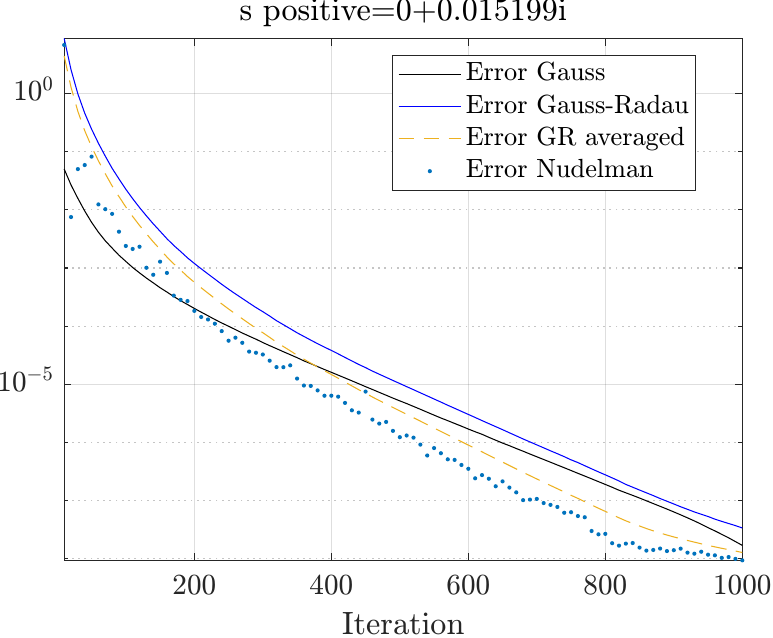}
 \caption{Convergence for an imaginary $s$ value.}
 \label{fig:EMImag}
 \end{subfigure}%
 \hspace{1em}
 \begin{subfigure}[b]{.47\linewidth}
 \centering
\includegraphics[width = \linewidth]{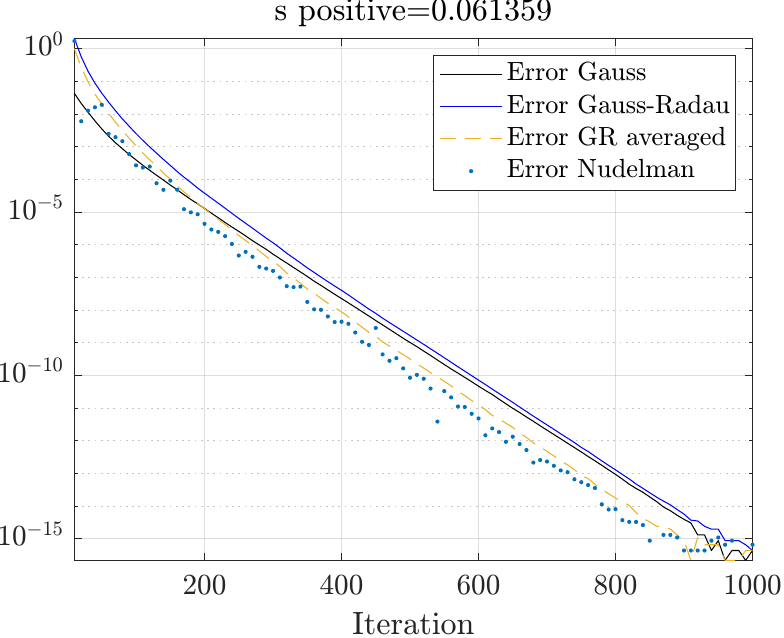}
 \caption{Convergence for a real $s$ value.}
 \label{fig:EMReal}
 \end{subfigure}
\caption{Converge result for the 3D electromagnetic case in the SISO setting ($p=1$).}
\label{fig:EMSweeps}
\end{figure}

\section{Conclusion}
We have shown that Kreĭn’s network representations of rational Stieltjes functions offer a powerful framework for interpreting the effect of truncating the Lanczos recursion. In particular, the truncation can be understood as inducing spurious reflections at the boundary of the computational domain, analogous to reflections in wave propagation problems.

We propose the introduction of absorbing boundary conditions inspired by the Kreĭn–Nudelman theory. These boundary modifications effectively suppress ``reflections" caused by the truncation of the Lanczos recursion. This improves  Lanczos-based approximations for matricees with dense spectra originating from discretization of a PDE operator on an unbounded domain. In particular, we demonstrate that changing the last diagonal element in the tridiagonal Lanczos matrix by a shift-dependent coefficient leads to substantial acceleration of convergence in quadrature computations. The acceleration is larger than averaging of Gau{\ss} and Gau\ss -Radau quadrature rules.

Additionally, we observe that the introduction of the square root of spectral parameter corresponds, in this context, to a Hermite–Padé approximation for functions with branch cuts. This connection aligns with recent developments in approximation theory \cite{Aptekarev,Suetin}, and highlights a promising direction for further investigation.

Finally, the extension of our approach to generalized quadratic forms of the type $B^T f(A) B$, {e.g.} exponential $\exp{(-At)}$ or Stieltjes functions $f$ is straightforward and falls naturally within the proposed framework.

Several promising directions for further work remain open. A more rigorous application and analysis of Hermite–Padé approximation in the presence of branch cuts is a natural next step, particularly for functions with branch cuts, and could deepen the theoretical foundation of the method. Enhancing adaptive parameter selection and extending the approach to rational Krylov subspaces and parametric model reduction are natural next steps. Incorporating AAA rational approximations may further improve the robustness and accuracy of the proposed framework. Demonstrating the Kreĭn-Nudelman approximations in the MIMO setting with $p>1$ is also a key objective. In addition, preliminary results in 3D suggest that closed-form expressions for $\phi,\varphi$ may be attainable, offering potential analytical simplifications.
%
\section*{Conflict of interest}
The authors declare that they have no conflict of interest.

\bibliographystyle{spmpsci} 
\bibliography{bib} 

\appendix
\section{Extraction of the Stieltjes parameters from }

From construction of $\Alpha_i$'s and $\Beta_i>0$ in Algorithm~\ref{alg:blockLanc} we obtain that $\hKappa_i$, $\gam_{i}$ are full rank, as long as no deflation occurs
in the block Lanczos recurrence, as $\Beta_i^T\Beta_i>0$ for $i=2,\ldots,m$, i.e. nonsingularity of all coefficients $\Alpha_i,\Beta_i$ ensures that 
Algorithm~\ref{alg:GammaExtraction} runs without breakdowns. 
 
 \begin{center}
\begin{minipage}{.55\linewidth}
 \begin{algorithm}[H]
 	\caption{Extraction algorithm $\gam/\hKappa$:\\ (Block $LDL^T$ Cholesky factorization of block tridiagonal $T_m$)}\label{alg:GammaExtraction}
 	\begin{algorithmic}
 	\normalsize
 		\State Given $\Alpha_i,\Beta_i$ and $\hKappa_1=I_p$
 		\State $\gam_1^{-1}=\hKappa_{1}^T \Alpha_1 \hKappa_{1}$ 		
		 \For{$i= 2,\dots, m$} 
 		\State $\hKappa_{i}^{-1}\quad	={ -} \gam_{i-1}(\hKappa_{i-1})^T\Beta_i^T \quad\, \phantom{-}(*)$
 		\State $\gam_i^{-1}\quad=\phantom{-} (\hKappa_{i})^T \Alpha_i \hKappa_{i}- \gam_{i-1}^{-1}\quad \,(\dagger)$
 		\EndFor

 	\end{algorithmic}
 	\label{alg:ExtractGam}
 \end{algorithm}
 \end{minipage}
 \end{center}
 Derivations of different variants of the Algorithm~\ref{alg:GammaExtraction} are given in \cite{zimmerling2025monotonicity,Druskin2016}.
 \section{Continued fraction interpretation}
 \subsection{Gau\ss ian quadrature}
 \begin{proposition}
 The Gau\ss ian quadrature can be written as a truncated matricial Stieltjes continued fraction (S-fraction) :
\be\label{eq:S-fraction1}
	\blF_m(s)= \cfrac{1}{s\hgam_1+ \cfrac{1}{\gam_1 + \cfrac{1}{s\hgam_2 +\cfrac{1}{ \ddots \cfrac{1}{s \hgam_m + \cfrac{1}{ \gam_m}} }}}} ,
\ee
where the basic building block of these S-fractions is given by the backward recursion
\be\label{eq:DefCFblock}
	\blC_i(s)=\cfrac{1}{s\hgam_i + \cfrac{1}{\gam_i+\blC_{i+1}(s)}}, \quad \blC_{m+1}(s) = 0,
\ee
with $\blC_i(s)\in\CC^{p\times p}$, $i=1,\ldots, m$ being s.p.d. for real positive $s$.

\end{proposition}
The proof for $p=1$ was known in Stieltjes work or even earlier, and its extension for $p\ge 1$ is given in \cite{zimmerling2025monotonicity}. Here, for completeness, we present its simplified variant.

 \begin{proof}
To prove this statement we use induction. At every step of the induction, we increase the continued fraction and solve a block tridiagonal pencil problem that increases by one block each step and coincides with the full pencil $(Z_m+s \widehat{\boldsymbol \Gamma}_m)$ at the last step. To facilitate this we define $\blUi{i}{j}$ as a family of $p \times p$ real matrices
\[
\{\blUi{i}{j} \in \mathbb{R}^{p \times p} | j=1,\dots,m \,\,, \,\, i=j,\dots,m\},
\]
where $i$ corresponds to the blocks coupled through the three-term recurrence relation defined by the pencil $(Z_m+s \widehat{\boldsymbol \Gamma}_m)$ and $j$ corresponds to the steps of the induction where at the $j$-th step we enforce the (block-Neumann) boundary condition 
\be\label{eq:BC}
-{\gam^{-1}_{j-1}{(\blUi{j}{j}-\blUi{j-1}{j})}}=I_p
\ee
This condition is enforced on the last $m-j$ blocks of the three term recurrence relation defined by the $m-j$ last block rows of $(Z_m+s \widehat{\boldsymbol \Gamma}_m)$ which can be written out as
\begin{align}
-{\gam^{-1}_{i}}{(\blUi{i+1}{j}-\blUi{i}{j})} +{\gam^{-1}_{i-1}{(\blUi{i}{j}-\blUi{i-1}{j})}} + s\hgam_i \blUi{i}{j} &=0, 	\quad i=j,\dots,m,\,\, j=1,\dots,m \label{eq:REC} \\
&{\quad} \blUi{m+1}{j}=0.
\end{align}
Note that for $j=1$ this corresponds together with \eqref{eq:BC} to the full pencil problem $(Z_m+s \widehat{\boldsymbol \Gamma}_m)[(\blUi{1}{1})^T, \dots, (\blUi{m}{1})^T ]^T=E_1$, which yields $\blF_m=\blUi{1}{1}$.\\

Further, we introduce the notation
\be
\blC_j = \blUi{j}{j}
\ee 
which to the condition \ref{eq:BC} provides the map
\be
-{\gam^{-1}_{j-1}{(\blUi{i}{j}-\blUi{i-1}{j})}} \overset{\blC_j}{\mapsto} \blUi{i}{j} \quad \forall i\geq j
\ee
since $\blU_{i\geq j}^{j}$ satisfies the same recursion $\forall j$. As shown later during the induction, these maps are positive definite and thus invertible and the inverse gives
\begin{align}
-{\gam^{-1}_{j-1}{(\blUi{i}{j}-\blUi{i-1}{j})}}&=\blC_j^{-1} \blUi{i}{j} \quad\forall i\geq j\\
\leftrightarrow \blUi{i-1}{j}&=(\blC_j+\gam_{j-1})\blC_j^{-1} \blUi{i}{j}
\end{align}
which for $i=j$ and assuming $\gam_{j}>0$ and invertibility of $\blC_j$ gives
\be\label{eq:ReccursionStart}
\blC_{j+1}^{-1} \blUi{j+1}{j}=(\blC_{j+1}+\gam_{j})^{-1}\blUi{j}{j}=-{\gam^{-1}_{j}}{(\blUi{j+1}{j}-\blUi{j}{j})}.
\ee

With this in place consider the recursion \eqref{eq:REC} for $i=j$, substitute the condition \eqref{eq:BC} and replace the term $-{\gam^{-1}_{j}}{(\blUi{j+1}{j}-\blUi{j}{j})}$ with the expression derived in \eqref{eq:ReccursionStart} to relate $\blC_{j+1}$ to $\blC_{j}$ as
\be
(\blC_{j+1}+\gam_{j})^{-1}\blC_j + s \hgam_j s \blC_j =I_p
\ee
which gives the basic building block of the material s-fraction as
\be\label{eq:SFracRecBlock}
\blC_j(s)=\cfrac{1}{s\hgam_j + \cfrac{1}{\gam_j+\blC_{j+1}(s)}} ,
	\qquad {\rm Re}\{s\}>0.
\ee

We are now ready to perform the induction, where we show that indeed $\blC$'s are invertible and that the upper recursion gives the desired S-Fraction for $\blF_m$

\begin{itemize}[label=$\lozenge$, itemsep=2ex]
		\item \emph{Base Case}: For the case $j=m$ the condition $\blUi{m+1}{m}=0$ implies $\blC_{m+1}=0$ in \eqref{eq:SFracRecBlock}. Alternatively Equation~ \ref{eq:REC} for $i=j=m$ can be written as 
\be
\gam_{j}^{-1}\blC_m + s \hgam_j s \blC_m =I_p
\ee
which starts the S-fraction as
\be
\blC_m(s)=\cfrac{1}{s\hgam_m + \cfrac{1}{\gam_m}} , \qquad {\rm Re}\{s\}>0.
\ee
which is positive definite and thus invertible.
	\item \emph{Induction Step}: 
	Next for $m>j \geq 1$ we assume that $\blC_{j+1}$ is positive definite such that
\[
\blC_j(s)=\cfrac{1}{s\hgam_j + \cfrac{1}{\gam_j+\blC_{j+1}(s)}} 
\]
extends the s-fraction and ensures that $\blC_j$ is positive definite and invertible. The recursion ends with $j=1$ at which point \eqref{eq:BC} and \eqref{eq:REC} coincide with the full pencil.		 
	\end{itemize}

\end{proof}
\subsection{Extension to Gau\ss -Radau and Kreĭn-Nudelman quadratures}\label{ap:ext}

To obtain the continuous fraction expression for $ {\hat\blF^{\phi,\varphi}}_m(s)$, we return to the Gau{\ss} quadrature.
Comparing (\ref{eqn:line1}-\ref{eqn:linem}) with \eqref{eq:DefCFblock} we obtain
$$
\blC_i^{-1} \blU_i:= - \frac 1 {\gam_{i-1}} (\blU_i-\blU_{i-1}) ,
$$
i.e., in the finite-difference interpretation they connect the solution with its differences $\blC_i$, , i.e., can be interpreted to the Dirichlet-to-Neumann operators.

Condition \ref{eqn:linem} yields the last condition of \eqref{eq:DefCFblock}. Following \cite{KreinNudelman1973,KreinNudelman1989}, and extending Dirichlet-to-Neumann interpretation, we replace it with the "impedance boundary condition " 
\be\label{eq:KN}
\blC_{m+1}^{-1} := \phi\sqrt{s}
\ee
thus we obtain
\be\label{eq:S-fraction2}
	\hat \blF_m^{\phi,\varphi} (s)= \cfrac{1}{s\hgam_1+ \cfrac{1}{\gam_1 + \cfrac{1}{s\hgam_2 +\cfrac{1}{ \ddots \cfrac{1}{s \hgam_m + \cfrac{1}{ \gam_m+(\varphi+\phi\sqrt{s})^{-1}}} }}}} .\ee
Similar to the limiting transitions written in terms of the $\hat T^{\phi,\varphi}_m(s)$, we can write in terms of the continued fraction
\be\label{eq:lim1}\lim_{{\phi,\varphi}\to \infty}\hat \blF^{\phi,\varphi}_m(s)= \blF_m(s)\ee and
\be\label{eq:lim2}
	\lim_{\phi,\varphi\to 0}\hat \blF^{\phi,\varphi}_m(s)= \cfrac{1}{s\hgam_1+ \cfrac{1}{\gam_1 + \cfrac{1}{s\hgam_2 +\cfrac{1}{ \ddots \cfrac{1}{s \hgam_m } }}}} = \blTF_m(s),\ee
 where $\blTF_m(s)$ is the block Gau\ss -Radau quadrature as defined in \cite{zimmerling2025monotonicity}.
 
 Below we prove Proposition~\ref{prop:main}.
 \begin{proof}
 To prove the first statement, we first notice that $\sqrt{s}$ is a two-valued Stieltjes function with the branch cut on $\RR_-$. Then $\blC_i$ are M\"obius transforms, so they recursively map complex plane to itself and a Stieltjes function to another Stieltjes function. Thus recursively, the branch cut of $\sqrt{s}$ is mapping to the branch cut with the same location.

 { The second statement is proved similarly to \eqref{bound:GR} in \cite{zimmerling2025monotonicity}. } We write counterpart of recursion \eqref{eq:DefCFblock} for $\hat \blF^\phi_m(s)$ as
\begin{eqnarray*}
\hat \blF^\phi_m(s)=\blC_{1}(s), \\ \blC_i(s)=\cfrac{1}{s\hgam_i + \cfrac{1}{\gam_i+\blC_{i+1}(s)}}, \quad i=m,m-1,\ldots, 1, \\ \blC_{m+1}(s) = (\varphi+\phi\sqrt{s})^{-1}. \end{eqnarray*} 
{ By construction $\blC_{m+1}$ is strictly monotonic with respect to $\phi^{-1}$ and $\varphi^{-1}$, and then $\blC_i$ are strictly monotonic with respect to $\blC_{i+1}$ for $i=m,m-1,\ldots, 1$, which gives strict monotonicity of $\hat \blF^\phi_m(s)$ with respect to $\phi^{-1}$ and $\varphi^{-1}$. The Gau{\ss} quadrature is obtained in \eqref{eq:lim1} as the limit for $\phi^{-1}\to 0$ and $\varphi^{-1}\to 0$ so it gives the lower bound, and \eqref{eq:lim2} as the limit for $\phi^{-1}\to \infty$ and $\varphi^{-1}\to \infty$, so we obtain the Gau\ss -Radau quadrature as the upper bound.}

\end{proof}

\begin{remark}
{To extend the first statement of Proposition~\ref{prop:main}to $p>1$ one needs to extend the M\"obius transform to matrix-valued coefficients. Even though such an extension seems valid, we have not found a literature reference and leave it for future work.

 The extension of the second statement of Proposition~\ref{prop:main} to $p>1$ can be obtained by applying the monotonicity reasoning to the matrix-valued argument. {Let} $\blF_m(s,\blC_{m+1})=\blC_{1}(s)$ {be the lower recursion truncated by $\blC_{m+1}$}
\begin{eqnarray*}
 \blC_i(s)=\cfrac{1}{s\hgam_i + \cfrac{1}{\gam_i+\blC_{i+1}(s)}}, \quad i=m,m-1,\ldots, 1,\end{eqnarray*}
 We consider three truncating values
 \be
 \blC^{\scalebox{0.5}[0.4]{\rm Gau\ss }}_{m+1}(s) = 0 < \blC^{\scalebox{0.5}[0.4]{\rm Nudelman}}_{m+1}(s) = (\varphi+\phi\sqrt{s})^{-1} < \blC^{\scalebox{0.5}[0.4]{\rm Gau\ss -Radau}}_{m+1}(s) = \infty . 
 \ee
From 
\be
\cfrac{1}{s\hgam_i + \cfrac{1}{\gam_i+\blC_{i+1}(s)}} < \cfrac{1}{s\hgam_i + \cfrac{1}{\gam_i+\blC_{i+1}(s)+\tau}}, \text{ for $\tau >0$}
\ee
it directly follows that
\begin{eqnarray*}
\blF_m(s,\blC^{\scalebox{0.5}[0.4]{\rm Gau\ss }}_{m+1}(s))< \blF_{m+1}(s,\blC^{\scalebox{0.5}[0.4]{\rm Gau\ss }}_{m+2}(s))< \blF_m(s,\blC^{\scalebox{0.5}[0.4]{\rm Nudelman}}_{m+1}(s))\\< 
\blF_{m+1}(s,\blC^{\scalebox{0.5}[0.4]{\rm Gau\ss -Radau}}_{m+2}(s)) <
\blF_m(s,\blC^{\scalebox{0.5}[0.4]{\rm Gau\ss -Radau}}_{m+1}(s)).
\end{eqnarray*}

}
\end{remark}

\end{document}

%% file: mathmacros.tex
\DeclareMathAlphabet{\itbf}{OML}{cmm}{b}{it}
\DeclareMathAlphabet\mathbfcal{OMS}{cmsy}{b}{n}

\renewcommand{\hat}{\widehat}
\renewcommand{\tilde}{\widetilde}
\def\RR{\mathbb{R}}
\def\bx{{{\itbf x}}}

\def\be{{\itbf e}}

\def\bU{{\itbf U}}

\def\bQ{{\itbf Q}}

\def\bV{{\itbf V}}

\def\bI{{\itbf I}}

\def\bD{{\itbf D}}

\def\bA{{\itbf A}}
\def\bB{{\itbf B}}
\def\bU{{\itbf U}}

\def\bI{{\itbf I}}

\def\bR{{\itbf R}}

\def\be#1\ee{%
    \begin{equation}%
    #1%
    \end{equation}%
}

\def\12{{\frac{1}{2}}}

\def\gam{{\boldsymbol \gamma}}
\def\Beta{{\boldsymbol \beta}}